\documentclass[12pt,letterpaper,reqno]{amsart}
\usepackage{amsmath}
\usepackage{amsfonts}
\usepackage{amssymb}
\usepackage{graphicx}
\usepackage{xcolor}

\newtheorem{theorem}{Theorem}
\newtheorem{lemma}[theorem]{Lemma}

\newtheorem{remark}{Remark}

\newtheorem{problem}[theorem]{Problem}

\allowdisplaybreaks

\begin{document}
	\author{Sergii Myroshnychenko, Kateryna Tatarko, and Vladyslav Yaskin}

	\address{S.~Myroshnychenko, Department of Mathematics and Statistics,  University of the Fraser Valley, Abbotsford, BC V2S 7M8,  Canada}
	
	\email{serhii.myroshnychenko@ufv.ca}

	\address{K.~Tatarko, Department of Pure Mathematics, University of Waterloo, Waterloo, ON N2L 3G1, Canada}
	
	\email{ktatarko@uwaterloo.ca}

	\address{V.~Yaskin, Department of Mathematical and Statistical Sciences, University of Alberta, Edmonton, AB T6G 2G1, Canada}
	
	\email{yaskin@ualberta.ca}

	\thanks{The authors were supported in part by NSERC}

	
	
\title[Centroids of sections]{Answers to questions of Gr\"unbaum and Loewner}

\begin{abstract} 
	We construct a convex body $K$  in $\mathbb{R}^n$, $n \geq 5$, with the property that there is exactly one hyperplane $H$ passing through $c(K)$, the centroid of $K$, such that the centroid of $K\cap H$ coincides with $c(K)$. This provides answers to questions of Gr\"unbaum and Loewner for $n\geq 5$. The proof is based on the existence of non-intersection bodies in these dimensions. 
\end{abstract}

\maketitle

\section{Introduction}

Let $K$ be a convex body in $\mathbb{R}^n$, i.e., a compact convex set with non-empty interior. The \textit{centroid} of $K$ (also known as the center of mass or barycenter) is the point
$$
c(K) = \frac{1}{|K|} \int_K x \, dx,
$$
where integration is with respect to Lebesgue measure and $|K|$ denotes the volume of $K$.

Consider the subset $\mathcal{G}$ of the interior of $K \subset \mathbb{R}^n$ with the property that each point $p$ from this set is the centroid of at least $n + 1$ different $(n-1)$-dimensional sections of $K$ through $p$. Gr\"unbaum 
\cite{Gr1} asked whether $\mathcal{G}$ is non-empty for every convex body $K$ and, in particular, whether the centroid $c(K)$ of $K$ belongs to $\mathcal{G}$; see also \cite{Gr2} and \cite[A8]{CFG}.

\begin{problem}\label{GrunbaumProblem} (Gr\"unbaum)
	Is the centroid $c(K)$ of $K \subset \mathbb{R}^n$ the centroid of at least $n + 1$ different $(n-1)$-dimensional sections of $K$ through $c(K)$?
\end{problem}

A more general question was asked by Loewner \cite[Problem 28]{F} a few years later.

\begin{problem}\label{LoewnerProblem} (Loewner)
	Let $\mu(K)$ be the number of hyperplane sections of $K$ passing through $c(K)$ whose centroid is the same as $c(K)$. Let $$\mu(n) =\displaystyle \min_{K\in \mathcal K^n} \mu(K),$$ where $\mathcal K^n$ is the class of all convex bodies in $\mathbb R^n$. What is the value of $\mu(n)$?

\end{problem}

As indicated by Gr\"unbaum \cite{Gr1} and Loewner \cite{F},  it is easy to show that Problem~\ref{GrunbaumProblem} has a positive answer in $\mathbb R^2$.
We include a proof in Section \ref{2D} for completeness. In \cite[\S6.2]{Gr2}, Gr\"unbaum claimed that $\mathcal{G} \neq \emptyset$ for every convex body $K\subset \mathbb R^n$, $n\ge 3$, however, Pat\'akov\'a, Tancer and Wagner \cite{PTW} discovered that one of the auxiliary statements in Gr\"unbaum's argument is incorrect. Additionally they proved that every convex body $K \subset \mathbb{R}^n$, $n \geq 3$, contains a point $p$ that is the centroid of at least four hyperplane sections through $p$. Thus, $\mathcal{G} \neq \emptyset$ for every convex body $K \subset \mathbb{R}^3$, while this is still unknown in dimensions $n \geq 4$.

Gr\"unbaum \cite{Gr1} has shown that every point in the interior of a convex body is the centroid of at least one section; see also \cite[p. 352]{MR}.  
Thus $\mu(n)\ge 1$ for all $n\ge 3$. 
In this paper we show that $\mu(n)=1$ for $n\ge 5$. In particular,  Problem~\ref{GrunbaumProblem} has a negative answer for $n \geq 5$.  We  construct a convex body of revolution $K$ satisfying $\mu(K)=1$ using Fourier analytic tools and exploiting the fact that there are origin-symmetric convex bodies that are non-intersection bodies in $\mathbb{R}^n$ for $n \geq 5$.

\section{Preliminaries}
We say that a compact set $K \subset \mathbb{R}^n$ is \textit{star-shaped} about the origin $0$ if for every point $x \in K$ each point of the interval $[0, x)$ is an interior point of $K$. The \textit{Minkowski functional} of $K$ is defined by
$$
\|x\|_K = \min\{\lambda \geq 0:  x \in \lambda K \}.
$$
 $K$ is called a \textit{star body} if it is compact, star-shaped about the origin and its Minkowski functional is a continuous function on $\mathbb{R}^n$.

The \textit{radial function} of a star body $K$ is defined by
$$
\rho_K(\xi) = \max \{\lambda > 0: \lambda \xi \in K \}, \quad \xi \in S^{n-1}.
$$
It is clear that $\rho_K(\xi) = \|\xi\|_K^{-1}$ for any $\xi$ on the unit sphere $S^{n-1}$ and $\rho_K$ is positive and continuous on $S^{n-1}$.

We say that $K$ is \textit{origin-symmetric} if $x\in K$ if and only if $-x \in K$. For an origin-symmetric star body $K$ we have $\rho_K(\xi)=\rho_K(-\xi) $ for all $\xi\in S^{n-1}$.

One of the main tools in this paper is the Fourier transform of distributions. The reader is referred to \cite{GS} and \cite{K2} for background information. Let $\mathcal S(\mathbb R^n)$ be the Schwartz space of rapidly decreasing infinitely differentiable functions on $\mathbb R^n$, called test functions. By $\mathcal S'(\mathbb R^n)$ we denote the space of distributions, i.e., continuous linear functionals on $\mathcal S(\mathbb R^n)$. For a test function $\varphi \in  \mathcal S(\mathbb R^n)$, its Fourier transform is given  by 
$$\widehat{\varphi}(x) = \int_{\mathbb R^n} \varphi(y) e^{-i \langle x,y\rangle}\, dy.$$
For a distribution $f\in  \mathcal S'(\mathbb R^n)$ its Fourier transform $\widehat{f}$ (also denoted by $f^\wedge$) is a distribution defined by
$$\langle \widehat{f}, \varphi\rangle = \langle {f}, \widehat{\varphi}\rangle,$$
for every $\varphi \in \mathcal S(\mathbb R^n)$. Note that for any even distribution $f$ we have $\widehat{\widehat f\, \,}=(2\pi)^n f$. 

If $f$ is a homogeneous function of degree $-k$ on $\mathbb R^n\setminus \{0\}$, where $0<k<n$, then $f$ can be thought of as a distribution acting on test functions by integration. In addition, if $f\in C^\infty (\mathbb R^n\setminus \{0\})$, then  $\widehat{f}$ is a distribution that acts on test functions as a homogeneous function of degree $-n+k$ which is   infinitely differentiable on  $\mathbb R^n\setminus \{0\}$; see \cite[Section 2.9]{E} or \cite[Lemma 3.6]{K2}. From now on, we will identify $\widehat{f}$ with this function. If $f$ is an even homogeneous function of degree $-n+1$, then its Fourier transform restricted to the sphere can be computed as follows (see \cite[Lemma 3.7]{K2}):
\begin{equation}\label{FR}
\widehat{f}(\xi) = \pi \int_{S^{n-1}\cap \xi^\perp} f(x) \, dx, \qquad \xi\in S^{n-1},
\end{equation}
where $\xi^\perp = \{x\in\mathbb R^n: \langle x,\xi\rangle=0\}$ is the hyperplane through the origin orthogonal to the vector $\xi$, and $\langle \cdot,\cdot \rangle$ stands for the standard inner product on $\mathbb{R}^n$.

We will denote the Euclidean norm of a vector  $x\in \mathbb R^n$ by $|x|$. 
The Fourier transform of $|x|^{-1}$ equals (see \cite[p.~363]{GS})
\begin{equation}\label{EuclNorm}
\left(|\cdot|^{-1}\right)^\wedge(x) = c_n |x|^{-n+1},
\end{equation}
where \begin{equation}\label{cn}
	c_n =  {2^{n-1} \pi^{\frac{n-1}2} \Gamma\left(\frac{n-1}2\right)} .\end{equation}

Note that if $T$ is an invertible linear transformation on $\mathbb R^n$, then 
\begin{equation}\label{LinTran}
(|Tx|^{-1})^\wedge(y) = c_n |\det T|^{-1} |T^{-t} y|^{-n+1},
\end{equation}
where $T^{-t}$ denotes the inverse of the transpose of $T$.

Let $f$ and $g$ be infinitely differentiable even functions on  $\mathbb R^n\setminus \{0\}$. Assume that $f$ and $g$ are homogeneous of degrees $-k$ and $-n+k$, respectively, for some $0<k<n$. Then  the following spherical Parseval formula holds:
\begin{equation}\label{Parseval}\int_{S^{n-1}} \widehat{f}(\xi) g(\xi)\, d\xi = \int_{S^{n-1}} f(\xi) \widehat{g}(\xi)\, d\xi.
\end{equation}
This formula is proven in \cite[Lemma 3.22]{K2} in the case when $f=\|\cdot\|_K^{-k}$ and $g=\|\cdot\|_L^{-n+k}$ for some star bodies $K$ and $L$, but this immediately yields the case of general functions. Alternatively, one can prove (\ref{Parseval}) using spherical harmonics; see \cite{GYY} for details.

The notion of the  intersection body of a star body was first introduced by Lutwak \cite{L} and has played an important role in convex geometry. 
Let $K$ and $L$ be origin-symmetric star bodies in $\mathbb R^n$.   We say that $K$ is the \textit{intersection body of} $L$ if 
$$\rho_{K}(\xi) = |L\cap \xi^\perp|,\quad \mbox{for all } \xi \in S^{n-1}.$$
Passing to polar coordinates, we see that $$\rho_{K}(\xi) = \frac{1}{n-1}\int_{S^{n-1}\cap \xi^\perp}\rho_L^{n-1}(x) \, dx.$$
By formula (\ref{FR}) this means that 
$$\|\xi\|_K^{-1} = \frac{1}{\pi(n-1)} \left(\|\cdot\|_L^{-n+1}\right)^\wedge(\xi),$$
which implies that $$\left(\|\cdot\|_K^{-1}\right)^\wedge(\xi) =\frac{(2\pi)^n}{\pi(n-1)} \ {\|\xi\|_L^{-n+1}} >0.$$

More generally, $K$ is called an \textit{intersection body} if $\rho_K$ is the spherical Radon transform (also known as the Minkowski-Funk transform) of a positive measure on $S^{n-1}$; see \cite{GLW}. Alternatively, the class of intersection bodies can be defined as the closure of the class of intersection bodies of star bodies in the radial metric. Intersection bodies were a key ingredient in the solution of the celebrated Busemann--Petty problem; the connection was discovered by Lutwak \cite{L}. 
The following Fourier analytic characterization of intersection bodies was obtained by Koldobsky \cite{K1}. An origin-symmetric star body in $\mathbb R^n$ is an intersection body if and only if $\|\cdot\|_K^{-1}$ is a positive definite distribution, i.e., the Fourier transform of $\|\cdot\|_K^{-1}$   is a positive distribution. All origin-symmetric convex bodies in dimensions 2, 3, and 4 are intersection bodies, while in dimensions $n\ge 5$ there are non-intersection bodies. Together with the Lutwak connection this gives a positive answer to the Busemann-Petty problem in dimensions $n\le 4$ and a negative answer in dimensions $n\ge 5$ (see \cite{K2} for the history of the problem and a Fourier analytic solution). For various constructions of non-intersection bodies see \cite{Ga}, \cite{Z}, \cite{K1},  \cite{GKS}, \cite{Ya}, \cite{Y}, \cite[Section 5.4.3]{R}.

\section{The planar case}\label{2D}

For $n = 2$,  Problem~\ref{GrunbaumProblem} has a positive answer.  In particular, the centroid of any triangle bisects three chords that are parallel to the sides. In general, the centroid $c(K)$ of any planar convex body $K \subset \mathbb{R}^2$ bisects at least three different chords of $K$. 
We provide a proof for completeness (see also \cite{B}).  Without loss of generality, we can choose the centroid of $K$ to be  the origin, that is,
\begin{equation*}
	\int_0^{2\pi} \rho^3(\theta) \cos(\theta) d\theta = 0 \quad \text{and} \quad  \int_0^{2\pi} \rho^3(\theta) \sin(\theta) d\theta = 0,
\end{equation*}
where the radial function $\rho$  of $K$ is given as a function of the polar angle $\theta$. Equivalently, 
\begin{equation}\label{eq::2d-centroid}
	\int_0^{\pi} f(\theta) \cos(\theta) d\theta = 0 \quad \text{and} \quad  \int_0^{\pi} f(\theta) \sin(\theta) d\theta = 0,
\end{equation}
where $f(\theta) = \rho^3(\theta) - \rho^3(\theta + \pi)$ is a continuous function on $\mathbb R$ with period  $2\pi$. Since $f(\theta) =- f(\theta+\pi)$,  there is at least one root $\theta_0\in [0, \pi)$ where  $f$ changes its sign. If such a root were unique then the function $f(\theta) \sin(\theta - \theta_0)$ would not change its sign in the interval $[0, \pi]$. However, this is impossible since  $\int_0^\pi f(\theta) \sin(\theta - \theta_0) d\theta = 0$ by \eqref{eq::2d-centroid}. Thus, $f$ has at least two distinct roots $\theta_0$ and $\theta_1$ ($\theta_0 < \theta_1$) on $[0, \pi
)$, where $f$ changes its sign. Note that $f$ cannot have exactly two such roots. If this were the case, $f$ would have the same sign on the intervals $(\theta_1-\pi, \theta_0)$ and $(\theta_1, \theta_0 + \pi)$ which is impossible since $f(\theta) = - f(\theta + \pi)$ for all $\theta$. Thus, $\rho(\theta) = \rho(\theta + \pi)$ for at least three distinct values in the interval $[0, \pi)$ which implies that $c(K)$ splits at least three chords in the corresponding directions in half.

\section{Main result}

\begin{theorem}\label{MainThm}
	There exists a convex body $K \subset \mathbb{R}^n$,  $n \geq 5$, with  centroid at the origin, such that 
	$$
	\langle c(K \cap \xi^{\perp}),e_n \rangle \geq 0 \quad \text{for any} \quad \xi \in S^{n-1},
	$$
	with equality if and only if $\xi = \pm e_n$, where $e_n = (0, \ldots,0,1)$.
\end{theorem}

We begin with an auxiliary construction. For a small  $a >0$, we define an origin-symmetric non-intersection convex body of revolution $M$ by
$$
\|x\|^{-1}_{M}= |x|^{-1} - 2 a^{n-2} \left(x_1^2 + \cdots +  x_{n-1}^2 + \frac{x_n^2}{a^2}\right)^{-\frac12}, \quad \text{for }   x \in \mathbb{R}^n\setminus\{0\}.
$$
This is a particular case of the construction from \cite[Theorem 2.2]{Y}.
We note that $|x|^{-1} > \left(x_1^2 + \cdots +  x_{n-1}^2 + \frac{x_n^2}{a^2}\right)^{-\frac12}$ for small $a >0$ and all $x \in \mathbb{R}^n\setminus\{0\}$, so $\|x\|^{-1}_{M} > 0$ on $ \mathbb{R}^n\setminus\{0\}$  and $M$ is well-defined.

\begin{lemma} \label{M_strictly_convex}
	$M$ is a   convex body   with strictly positive curvature when $a$ is sufficiently small.
\end{lemma}
\begin{proof}	This is a standard perturbation argument; cf. \cite[p.~96]{K2}.
	To show that $M$ has  strictly positive curvature, observe that $M$ is a small perturbation of the Euclidean ball whose curvature is 1 everywhere. We can control the first and second derivatives of the perturbation function $2 a^{n-2} \left(x_1^2 + \cdots +  x_{n-1}^2 + \frac{x_n^2}{a^2}\right)^{-\frac12}$. These derivatives are of order $a^{n-4}$ which are  small for sufficiently small $a>0$ (since $n\geq 5$). Hence, $M$ has strictly positive curvature.
\end{proof}

From now on we will fix a value of $a$ for which Lemma \ref{M_strictly_convex} is true.

\begin{lemma} \label{defM}
$M$ is not an intersection body. In particular,	$\left(\|\cdot\|^{-1}_{M}\right)^\wedge(\pm e_n) < 0$.
\end{lemma}
\begin{proof}

To compute the Fourier transform of $\|x\|^{-1}_{M}$ we use formulas (\ref{EuclNorm}) and (\ref{LinTran}):
$$
\left(\|\cdot\|^{-1}_{M}\right)^\wedge(x) = c_{n} \left( |x|^{-n+1} - 2a^{n-1} \left(x_1^2 + \cdots +  x_{n-1}^2 + a^2 x_n^2\right)^{\frac{-n+1}2}\right),
$$
where $c_n$ is the positive constant defined in \eqref{cn}.
It remains to observe that $$\left(\|\cdot\|^{-1}_{M}\right)^\wedge(\pm e_n) = - c_n < 0.$$
\end{proof}

\noindent {\it Proof of Theorem \ref{MainThm}.}
By Lemma~\ref{defM}, $\left(\|\cdot\|^{-1}_M\right)^{\wedge}(e_n) < 0$. Since $\left(\|\cdot\|^{-1}_M\right)^{\wedge}$ is continuous on the sphere $S^{n-1}$, there is    an open spherical ball  $\Omega(e_n) \subset S^{n-1}$  centered at $e_n$ such that $\left(\|\cdot\|^{-1}_M\right)^{\wedge}(\xi) < 0$ for any $\xi \in \Omega(e_n).$ Define $\Omega(-e_n)$ to be the reflection of $\Omega(e_n)$ with respect to the center of the sphere. Let  $G \in C^{\infty}(S^{n-1})$ be an even function that is invariant under rotations about the $x_n$-axis and that satisfies the following properties:
$$
G(\xi)  = \begin{cases}
\text{positive}, \quad \xi \in \Omega(\pm e_n) \backslash \{\pm e_n\};\\
0, \quad \xi \in \{\pm e_n\} \cup S^{n-1} \backslash \Omega(\pm e_n).
\end{cases}
$$
By construction,
\begin{equation}\label{IntG}
\int_{S^{n-1}} \left(\|\cdot\|^{-1}_M\right)^{\wedge}(\xi) G(\xi) d\xi < 0.
\end{equation}
Next, we define the following function $H \in C^{\infty}(S^{n-1})$ that is invariant under rotations about the $x_n$-axis:
$$
H(x) = |x|^{-1} - (4(x_1^2 + \dots + x_{n-1}^2) + x_n^2)^{-\frac12}.
$$
Note that $H(x) >0$   when $x\in S^{n-1}\setminus\{\pm e_n\}$, and $H(\pm e_n) = 0$.  Extending $H$ to $\mathbb R^n$ as a homogeneous function of degree $-1$ and using formulas (\ref{EuclNorm}) and (\ref{LinTran}), we obtain

\begin{align*}
	\widehat{H}(x) &= c_n \left(|x|^{-n+1} - \frac1{2^{n-1}} \left(\frac{x_1^2 + \dots + x_{n-1}^2}4 + x_n^2\right)^{\frac{-n+1}2}\right)\\
	&=c_n \left(|x|^{-n+1} - \left(x_1^2 + \dots + x_{n-1}^2 + 4x_n^2\right)^{\frac{-n+1}2}\right).
\end{align*}
 Observe that $\widehat{H} (x) \ge 0$ for $x\in \mathbb R^n\setminus\{0\}$.  
Hence, using the spherical Parseval formula (\ref{Parseval}), we get
\begin{equation}\label{IntH}
 \int_{S^{n-1}} \left(\|\cdot\|^{-1}_M\right)^{\wedge}(\xi) H(\xi) \, d\xi =  \int_{S^{n-1}} \|\xi\|^{-1}_M \widehat{H}(\xi)\, d\xi > 0.
\end{equation}
Finally, for $\lambda\in [0,1]$   we define 
$$
g_\lambda = (1-\lambda)G + \lambda H.
$$
Since $G(x)\ge 0$ and $H(x)>0$ for all $\xi\in S^{n-1}\setminus\{\pm e_n\}$, and $G(\pm e_n)=H(\pm e_n)=0$, we 
see that for all $\lambda\in (0,1]$ the function $g_{\lambda}$ satisfies the  properties: $g_{\lambda}(\xi) > 0$ for  all $\xi\in S^{n-1}\setminus\{\pm e_n\}$, and $g_\lambda(\pm e_n) = 0$.

Since $G$ and $H$ are $C^\infty$ functions of the sphere, so is $g_{\lambda}$. Extending $g_{\lambda}$ to  $ \mathbb R^n\setminus\{0\}$ as a homogeneous function of degree $-1$ and denoting  the Fourier transform of this extension by $\widehat{g_\lambda}$, we define a function $\phi_\lambda$ on $S^{n-1}$ by the formula
$$\phi_\lambda(\xi) = \frac{1}{\xi_n} \widehat{g_\lambda}(\xi), \qquad \xi \in S^{n-1},$$
where $\xi_n = \langle \xi, e_n \rangle$.
Note that $\phi_\lambda$ does not have a singularity when $\xi_n=0$, since $\widehat{g_\lambda}(\xi)=0$ for all $\xi\in e_n^\perp$. To show the latter, observe that $\widehat{g_\lambda}$  is rotationally invariant as $g_\lambda$ has the same property. Therefore for $\xi\in e_n^\perp$ we have
\begin{align*}
	\widehat{g_{\lambda}}(\xi)& =   \frac{1}{|S^{n-2}|}\int_{S^{n-1} \cap e_n^{\perp}} \widehat{g_{\lambda}}(\xi) \, dx = \frac{1}{|S^{n-2}|}\int_{S^{n-1} \cap e_n^{\perp}} \widehat{g_{\lambda}}(x) \, dx \\
	&= \frac{1}{\pi |S^{n-2}|} \, \widehat{\widehat{g_{\lambda}}}(e_n) = \frac{(2\pi)^n}{\pi |S^{n-2}|}\, g_\lambda(e_n)=0,
\end{align*}
where we used formula (\ref{FR}).

Since $\widehat{g_{\lambda}}$ is even, infinitely smooth, and vanishes on $e_n^\perp$, we can define $\phi_\lambda$ to be zero on $e_n^\perp$. The  function   $\phi_\lambda$  constructed this way is odd and infinitely smooth. Also observe that the first and second derivatives of $\widehat{g_{\lambda}}$ are bounded on the sphere, uniformly in $\mathbb \lambda\in[0,1]$. This is due to the fact that $\phi_\lambda$ is linear in $\lambda$.

For a small $\varepsilon>0$ define a body $K$ as follows:
$$\rho_K(\xi) = \left(\rho_M^n(\xi)+ \varepsilon \phi_\lambda (\xi)\right)^{1/n}, \qquad \xi\in S^{n-1}.$$
Since $M$ has strictly positive curvature and the derivatives of $\phi_\lambda$ are bounded on the sphere, uniformly in $\lambda$, we conclude that there is a small $\alpha>0$ such that the body $K$ is convex  for all  $\varepsilon \in [0,\alpha]$ and all $\lambda \in [0,1]$. Additionally, $K$ is a body of revolution about the $x_n$-axis.

Now observe that the centroids of sections $K \cap \xi^{\perp}$ all lie in the open half-space $\{x \in \mathbb{R}^n: \, x_n > 0\}$, unless $\xi=\pm e_n$. Indeed, for all $\xi\in S^{n-1}\setminus\{\pm e_n\}$ we have
\begin{align*}
|K \cap \xi^{\perp}| \,\langle c(K\cap \xi^{\perp}),e_n\rangle&  = \int_{K \cap \xi^{\perp}} x_n\, dx\\ 
&=\frac1n\int_{S^{n-1} \cap \xi^{\perp}} x_n   \rho_K^n(x)  \, dx \\
&=\frac1n\int_{S^{n-1} \cap \xi^{\perp}} x_n \left (\rho_M^n(x)+ \varepsilon \phi_\lambda (x) \right) \, dx \\
& =\frac 1n \int_{S^{n-1} \cap \xi^{\perp}} x_n \varepsilon \phi_\lambda(x) \, dx \\
&= \frac{\varepsilon}{\pi n}\left(x_n   \phi_\lambda(x)\right)^\wedge (\xi) = \frac{\varepsilon (2\pi)^n}{\pi n} g_\lambda(\xi)> 0,
\end{align*}
where we used the origin-symmetry of $M$ and formula (\ref{FR}). Note that the latter strict inequality holds for all $\varepsilon>0$
 and all $\lambda\in (0,1]$. 

Now we will choose $\lambda$ and $\varepsilon$ so that  the centroid of $K$ is at the origin. Since $K$ is a body of revolution, it is enough to ensure  that the quantity below can  be equal to  zero.
\begin{align*}
	|K| \, \langle c(K), e_n \rangle & = \int_K x_n \, dx= \frac{1}{n+1} \int_{S^{n-1}} x_n \rho_K^{n+1}(x) \, dx \\
	& = \frac{1}{2(n+1)} \int_{S^{n-1}} x_n\left(\rho_K^{n+1}(x)-\rho_K^{n+1}(-x)\right) \, dx\\
	&=\frac{1}{2(n+1)} \int_{S^{n-1}} \left(\frac{\rho_K^{n+1}(x)-\rho_K^{n+1}(-x)}{\varepsilon \phi_\lambda(x) }\right)   x_n\varepsilon \phi_\lambda(x) \, dx.
\end{align*}

Since $\phi_\lambda$ is odd and $M$ is origin-symmetric, we have
\begin{align*}
\frac{\rho_K^{n+1}(x)-\rho_K^{n+1}(-x)}{\varepsilon \phi_\lambda(x) } &= \frac{\left(\rho_M^n(x)+\varepsilon \phi_\lambda(x) \right)^{\frac{n+1}{n}}-\left(\rho_M^n(-x)+\varepsilon \phi_\lambda(-x) \right)^{\frac{n+1}{n}}}{\varepsilon \phi_\lambda(x) }\\
&=\frac{\left(\rho_M^n(x)+\varepsilon \phi_\lambda(x) \right)^{\frac{n+1}{n}}-\left(\rho_M^n(x)-\varepsilon \phi_\lambda(x) \right)^{\frac{n+1}{n}}}{\varepsilon \phi_\lambda(x) }\\
&=\rho_M^{n+1} (x)\frac{\left(1+\varepsilon \phi_\lambda(x)\rho_M^{-n}(x) \right)^{\frac{n+1}{n}}-\left(1-\varepsilon \phi_\lambda(x) \rho_M^{-n}(x)\right)^{\frac{n+1}{n}}}{\varepsilon \phi_\lambda(x) }\\
&=\rho_M^{n+1} (x)\frac{\frac{2(n+1)}{n}\varepsilon \phi_\lambda(x) \rho_M^{-n}(x) +\varepsilon^3 \phi_\lambda^3(x) R_1(x)}{\varepsilon \phi_\lambda(x)  }\\
&=  \frac{2(n+1)}{n} \left(\rho_M (x) +\varepsilon^2   R (x) \right),
\end{align*}
where $R_1$ is  the  remainder in the Taylor expansion in $\varepsilon$ of the corresponding expression above, and $R(x)= \frac{n}{2(n+1)}\rho_M^{n+1} (x)\phi_\lambda^2(x) R_1(x)$.

Thus,
\begin{align*}
 |K| \, \langle c(K), e_n \rangle &=\frac{1}{n} \int_{S^{n-1}} \left( \rho_M (x) +\varepsilon^2   R (x)\right)  x_n\varepsilon \phi_\lambda(x) \, dx \\
 &=\frac{\varepsilon}{n} \int_{S^{n-1}} \left( \rho_M (x) +\varepsilon^2   R (x)\right) \widehat{g_\lambda}(x) \, dx \\
  &=\frac{\varepsilon}{n} \int_{S^{n-1}} \left( \widehat{\rho_M} (\xi) +\varepsilon^2\widehat{R} (\xi)\right) {g_\lambda}(\xi) \, d\xi,
\end{align*}
where we used the spherical Parseval formula (\ref{Parseval}).

Define
$$F(\lambda,\varepsilon)=\int_{S^{n-1}} \left( \widehat{\rho_M} (\xi) +\varepsilon^2\widehat{R} (\xi)\right) {g_\lambda}(\xi) \, d\xi.$$
Note that all derivatives of $R$ depend continuously on $\lambda$ and $\varepsilon$. Therefore, by \cite[Lemma 3.16]{K2}, $\widehat{R}$ is also a continuous function of $\lambda$ and $\varepsilon$, and hence so is 
 $F$ for $(\lambda,\varepsilon)\in [0,1]\times [0,\alpha].$

Observe that  $$F(\lambda,0)=\int_{S^{n-1}} \widehat{\rho_M}(\xi)  g_{\lambda}(\xi) \, d\xi$$
is a sign-changing function of  $\lambda\in [0,1]$. Indeed, 
$$F(0,0)=\int_{S^{n-1}}  \widehat{\rho_M}(\xi)   G(\xi) \, d\xi<0$$
and 
$$F(1,0)=\int_{S^{n-1}}  \widehat{\rho_M}(\xi)  H(\xi) \, d\xi>0$$
by (\ref{IntG}) and (\ref{IntH}) respectively.

Since $F(\lambda, \varepsilon)$ is a continuous function on $[0,1] \times [0, \alpha]$, there exist a small $\varepsilon_0>0$ and $\lambda_0 \in (0,1)$ such that $F(\lambda_0, \varepsilon_0) = 0$, as required.

\qed

\begin{remark}
	The above construction does not work when $n=3$ and $n=4$ since there are no non-intersection bodies in these dimensions. We suspect that $\mu(3)=\mu(4)=1$, but the construction has to be more delicate.
	\end{remark}
	
\begin{remark}
	It is natural to ask about the analogues of Problems \ref{GrunbaumProblem} and \ref{LoewnerProblem} for $k$-dimensional sections, $1 \leq k \leq n-2$. However, for any convex body \( K \subset \mathbb{R}^n \), \( n \geq~3 \), there exist infinitely many \( k \)-dimensional sections of \( K \) whose centroids coincide with \( c(K) \). Indeed, if there were a convex body \( K \) with only finitely many such \( k \)-dimensional sections, then there would exist a \( (k+1) \)-dimensional affine subspace \( V \) passing through \( c(K) \) that does not contain any of these sections. According to \cite{Gr1}, \( K \cap V \) would then have a \( k \)-dimensional section with centroid at \( c(K) \), leading to a contradiction.

\end{remark}

	\begin{remark}
This work continues a recent trend of studying relative positions of various centroids; see e.g., \cite{MTY} and \cite{NRY}. In \cite{NRY} the authors found an optimal upper bound on the distance between the centroid of a planar body and its boundary. The corresponding question is open in $\mathbb R^n$, $n\ge 3$. 
In \cite{MTY} the authors found an optimal upper bound on the distance between a projection of the centroid of a convex body $K$ and the centroid of the projection of $K$. An analogue of the latter question for sections  is currently open. It was stated in \cite{S} and  motivated by Gr\"unbaum-type inequalities; see \cite{Gr}, \cite{FMY}, \cite{SZ}, \cite{MNRY}, \cite{MSZ}.

	\end{remark}

\end{document}